\documentclass[graybox]{svmult}

\usepackage{mathptmx}       
\usepackage{helvet}         
\usepackage{courier}        
\usepackage{type1cm}        
\usepackage{makeidx}         
\usepackage{graphicx}        
\usepackage{multicol}        
\usepackage[bottom]{footmisc}

\usepackage{amsmath}
\usepackage{amssymb}
  \newtheorem{fact}[theorem]{Fact}
  \newtheorem{pcriterion}[theorem]{Properness Criterion}
\newtheorem{thmalph}{Theorem}

\newtheorem{probalph}[thmalph]{Problem}

\begin{document}

\title*{Intrinsic sound of anti-de Sitter manifolds}
\author{Toshiyuki Kobayashi}
\institute{Toshiyuki Kobayashi \at 
Kavli IPMU and Graduate School of Mathematical Sciences, 
The University of Tokyo, 
\email{toshi@ms.u-tokyo.ac.jp}
}
\maketitle

\abstract{As is well-known 
 for compact Riemann surfaces,
 eigenvalues of the Laplacian
 are distributed discretely 
 and most of eigenvalues vary
 viewed as functions on the Teichm{\"u}ller space.  
We discuss a new feature 
 in the Lorentzian geometry, 
 or more generally,
 in pseudo-Riemannian geometry.  
One of the distinguished features is
 that $L^2$-eigenvalues
 of the Laplacian may be distributed densely in ${\mathbb{R}}$
 in pseudo-Riemannian geometry.  
For three-dimensional anti-de Sitter manifolds, 
we also explain another feature 
 proved in joint with F. Kassel
 [Adv. Math. 2016]
 that there exist countably many $L^2$-eigenvalues
 of the Laplacian 
 that are stable under any small deformation
 of anti-de Sitter structure.  
Partially supported by
        Grant-in-Aid for Scientific Research (A) (25247006), Japan
        Society for the Promotion of Science.  
}

\noindent
\textit{Keywords and phrases:}
Laplacian, 
locally symmetric space,
Lorentzian manifold,
spectral analysis, 
Clifford--Klein form, 
reductive group,
 discontinuous group
\medskip
\noindent
\textit{2010 MSC:}
Primary
22E40, 22E46, 58J50:
Secondary
 11F72, 53C35

\section{Introduction}
\label{sec:1}

Our \lq\lq{common sense}\rq\rq\
for music instruments says:

\[
\mbox{\lq\lq{shorter strings produce a higher pitch
 than longer strings}\rq\rq, }
\]
\[\mbox{\lq\lq{thinner strings produce a higher pitch
 than thicker strings}\rq\rq.}\]

Let us try to 
 \lq\lq{hear the sound
 of pseudo-Riemannian locally symmetric spaces}\rq\rq.  
Contrary to our \lq\lq{common sense}\rq\rq\
 in the Riemannian world, 
 we find a phenomenon
 that compact three-dimensional anti-de Sitter manifolds
 have \lq\lq{intrinsic sound}\rq\rq\
 which is stable under any small deformation.  
This is formulated
 in the framework of spectral analysis
 of anti-de Sitter manifolds, 
 or more generally,
 of pseudo-Riemannian locally symmetric spaces $X_{\Gamma}$.  
In this article, 
we give a flavor
 of this new topic
 by comparing it with the flat case
 and the Riemannian case.

To explain briefly the subject,
 let $X$ be a pseudo-Riemannian manifold,
  and $\Gamma$ a discrete isometry group
 acting properly discontinuously and freely
 on $X$.  
Then the quotient space $X_{\Gamma}:=\Gamma \backslash X$
 carries a pseudo-Riemannian manifold structure
 such that the covering map $X \to X_{\Gamma}$ is isometric.  
We are particularly interested in the case
 where $X_{\Gamma}$ is a pseudo-Riemannian locally symmetric space, 
 see Section \ref{subsec:3B}.

Problems we have in mind are symbolized
 in the following diagram:

\begin{table}[h]
\begin{center}
\begin{tabular}{c|c|c}
&existence problem & deformation {\it{v.s.}} rigidity
\\
\hline
Geometry
&
Does cocompact $\Gamma$ exist? 
&
Higher Teichm{\"u}ller theory {\it{v.s.}} rigidity theorem
\\
& (Section \ref{subsec:4A})
& (Section \ref{subsec:4B})
\\
\hline
Analysis
& Does $L^2$-spectrum exist?
& Whether $L^2$-eigenvalues vary or not
\\
&
(Problem A)
&
(Problem B)
\end{tabular}
\end{center}
\end{table}%

\section{A program}
\label{sec:2}
In \cite{KK11, Adv16, kob09}
 we initiated the study 
 of 
\lq\lq{spectral analysis on pseudo-Riemannian locally symmetric spaces}\rq\rq\
 with focus on the following two problems:
\begin{probalph}
\label{prob:A}
Construct eigenfunctions
 of the Laplacian $\Delta_{X_{\Gamma}}$ on $X_{\Gamma}$.  
Does there exist
a nonzero $L^2$-eigenfunction?
\end{probalph}

\begin{probalph}
\label{prob:B}
Understand the behaviour
 of $L^2$-eigenvalues
 of the Laplacian $\Delta_{X_{\Gamma}}$ on $X_{\Gamma}$
 under small deformation
 of $\Gamma$ inside $G$.  
\end{probalph}

Even when $X_{\Gamma}$ is compact,
 the existence
 of countably many $L^2$-eigenvalues
 is already nontrivial 
 because the Laplacian $\Delta_{X_{\Gamma}}$ is not elliptic
 in our setting.  
We shall discuss in Section \ref{subsec:2B}
 for further difficulties
 concerning Problems \ref{prob:A} and \ref{prob:B}
 when $X_{\Gamma}$ is non-Riemannian.

We may extend these problems
 by considering {\it{joint}} eigenfunctions
 for \lq\lq{invariant differential operators}\rq\rq\
 on $X_{\Gamma}$
 rather than the single operator $\Delta_{X_{\Gamma}}$.  
Here by \lq\lq{invariant differential operators
on ${X_{\Gamma}}$}\rq\rq\
 we mean differential operators 
 that are induced from $G$-invariant ones
 on $X=G/H$.  
In Section \ref{sec:7}, 
 we discuss Problems \ref{prob:A} and \ref{prob:B}
 in this general formulation
 based on the recent joint work \cite{Adv16, KKspec} with F. Kassel.  

\subsection{Known results}
\label{subsec:2A}

Spectral analysis
 on a pseudo-Riemannian locally symmetric space 
$
   X_{\Gamma} = \Gamma \backslash X=\Gamma \backslash G/H
$
 is already deep 
 and difficult
 in the following special cases:
\begin{enumerate}
\item[1)]
(noncommutative harmonic analysis
 on $G/H$)
$\Gamma =\{e\}$.  

In this case,
 the group $G$ acts unitarily
 on the Hilbert space $L^2(X_{\Gamma})=L^2(X)$
 by translation
$
  f(\cdot) \mapsto f(g^{-1}\cdot)
$, 
and the irreducible decomposition of $L^2(X)$
 ({\it{Plancherel-type formula}})
is essentially equivalent to the spectral analysis
 of $G$-invariant differential operators
 when $X$ is a semisimple symmetric space.  
Noncommutative harmonic analysis
 on semisimple symmetric spaces $X$
 has been developed extensively by the work 
 of Helgason, Flensted-Jensen, 
Matsuki--Oshima--Sekiguchi, 
Delorme, 
van den Ban--Schlichtkrull
 among others as a generalization
 of Harish-Chandra's earlier work
 on the regular representation $L^2(G)$
 for group manifolds.  
\item[2)]
(automorphic forms)
$H$ is compact and $\Gamma$ is arithmetic.  

If $H$ is a maximal compact subgroup of $G$, 
 then $X_{\Gamma}=\Gamma \backslash G/H$ is a Riemannian locally symmetric space
 and the Laplacian $\Delta_{X_{\Gamma}}$ is an elliptic differential operator.  
Then there exist infinitely many $L^2$-eigenvalues
 of $\Delta_{X_{\Gamma}}$
 if $X_{\Gamma}$ is compact 
 by the general theory
 for compact Riemannian manifolds
 (see Fact \ref{fact:RLop}).  
If furthermore $\Gamma$ is irreducible,
 then Weil's local rigidity theorem 
\cite{Weil} states
that nontrivial deformations exist
 only when $X$ is the hyperbolic plane
 $SL(2,{\mathbb{R}})/SO(2)$, 
 in which case compact quotients $X_{\Gamma}$
 have a classically-known deformation space modulo
 conjugation, 
 {\it{i.e.}}, 
their Teichm{\"u}ller space.  
Viewed as a function
 on the Teichm{\"u}ller space, 
 $L^2$-eigenvalues vary analytically
\cite{BC, W}, 
 see Fact \ref{fact:6.2}.

Spectral analysis on $X_{\Gamma}$
 is closely related to the theory of automorphic forms
 in the Archimedian place
 if $\Gamma$ is an arithmetic subgroup. 

\item[3)]
(abelian case)
$G={\mathbb{R}}^{p+q}$
 with $H=\{0\}$
 and $\Gamma = {\mathbb{Z}}^{p+q}$.  

We equip $X=G/H$
 with the standard flat pseudo-Riemannian structure
 of signature $(p,q)$
 (see Example \ref{ex:Rpq}).  
In this case, 
$G$ is abelian, 
 but $X=G/H$ is non-Riemannian.  
This is seemingly easy, 
 however,
 spectral analysis
 on the $(p+q)$-torus
 ${\mathbb{R}}^{p+q}/{\mathbb{Z}}^{p+q}$
 is much involved, as we shall observe a connection with Oppenheim's conjecture
 (see Section \ref{subsec:5B}).  
\end{enumerate}

\subsection{Difficulties in the new settings}
\label{subsec:2B}

If we try to attack a problem
 of spectral analysis on $\Gamma \backslash G/H$
 in the more general case
 where $H$ is noncompact
 and $\Gamma$ is infinite, 
then new difficulties 
 may arise from several points of view:
\begin{enumerate}
\item[(1)]
Geometry. \enspace
The $G$-invariant pseudo-Riemannian structure on $X=G/H$ 
 is not Riemannian anymore,
 and discrete groups
 of isometries of $X$
 do not always act properly discontinuously
 on such $X$.  
\item[(2)]
Analysis. \enspace
The Laplacian $\Delta_X$ on $X_{\Gamma}$
 is not an elliptic differential operator.  
Furthermore, 
 it is not clear 
 if $\Delta_X$ has a self-adjoint extension 
 on $L^2(X_{\Gamma})$.  
\item[(3)]
Representation theory. \enspace
If $\Gamma$ acts properly discontinuously on $X=G/H$
 with $H$ noncompact, 
 then the volume of $\Gamma \backslash G$ is infinite, 
 and the regular representation $L^2(\Gamma \backslash G)$
 may have infinite multiplicities.  
In turn,  
 the group $G$ may not have a good control
 of functions on $\Gamma \backslash G$.  
Moreover $L^2(X_{\Gamma})$
 is not a subspace of $L^2(\Gamma \backslash G)$
 because $H$ is noncompact.  
All these observations suggest 
that an application of the representation theory
 of $L^2(\Gamma \backslash G)$
 to spectral analysis on $X_{\Gamma}$
is rather limited 
 when $H$ is noncompact.  
\end{enumerate}

Point (1) creates some underlying difficulty to Problem \ref{prob:B}:
we need to consider locally symmetric spaces $X_{\Gamma}$
 for which proper discontinuity of the action
 of $\Gamma$ on $X$
 is preserved 
 under small deformations
 of $\Gamma$ in $G$.  
This is nontrivial.  
This question was first studied 
 by the author \cite{kob93, kob98}. 
See \cite{Ka12} for further study. 
An interesting aspect of the case of noncompact $H$
 is that there are more examples where nontrivial deformations
 of compact quotients exist than for compact $H$
({\textit{cf}}. Weil's local rigidity theorem \cite{Weil}).  
Perspectives from Point (1) will be discussed in Section \ref{sec:4}.

Point (2) makes Problem \ref{prob:A} nontrivial.  
It is not clear 
 if the following well-known properties
 in the {\it{Riemannian}} case holds
 in our setting in the {\it{pseudo-Riemannian}} case.  

\begin{fact}
\label{fact:RLop}
Suppose $M$ is a compact Riemannian manifold.  
\begin{enumerate}
\item[{\rm{(1)}}]
The Laplacian $\Delta_M$ extends
 to a self-adjoint operator on $L^2(M)$.  
\item[{\rm{(2)}}]
There exist infinitely many $L^2$-eigenvalues of $\Delta_M$.  
\item[{\rm{(3)}}]
An eigenfunction of $\Delta_M$ is infinitely differentiable.  
\item[{\rm{(4)}}]
Each eigenspace of $\Delta_M$ is finite-dimensional.  
\item[{\rm{(5)}}]
The set of $L^2$-eigenvalues
 is discrete in ${\mathbb{R}}$. 
\end{enumerate}
\end{fact}

\begin{remark}
\label{remm:RLop}
We shall see 
 that  the third to fifth properties 
 of Fact \ref{fact:RLop} may fail 
 in the pseudo-Riemannian case,
 {\it{e.g.}}, 
 Example \ref{ex:n2} for (3) and (4), 
 and $M={\mathbb{R}}^{2,1}/{\mathbb{Z}}^3$
 (Theorem \ref{thm:Oppen}) for (5).  
\end{remark}

In spite of these difficulties,
 we wish to reveal a mystery of spectral analysis
 of pseudo-Riemannian locally homogeneous spaces
 $X_{\Gamma}=\Gamma \backslash G/H$.  
We shall discuss self-adjoint extension
 of the Laplacian 
 in the pseudo-Riemannian setting
 in Theorem \ref{thm:7.2}, 
 and the existence of countable many $L^2$-eigenvalues
 in Theorems \ref{thm:6.1}, 
 \ref{thm:7.1} and \ref{thm:7.2}.  

\section{Pseudo-Riemannian manifold}
\label{sec:3}

\subsection{Laplacian on pseudo-Riemannian manifolds}
\label{subsec:3A}

A {\it{pseudo-Riemannian manifold}} $M$ is a smooth manifold
 endowed with a smooth, nondegenerate, 
 symmetric bilinear tensor $g$
 of signature $(p,q)$
 for some $p$, $q$ $\in {\mathbb{N}}$.  
$(M,g)$ is a Riemannian manifold
 if $q=0$, 
 and is a Lorentzian manifold if $q=1$.  
The metric tensor $g$ induces a Radon measure $d \mu$ on $X$, 
 and the divergence $\operatorname{div}$.  
Then the Laplacian 
\[
  \Delta_M := \operatorname{div} \operatorname{grad}, 
\]
is a differential operator
 of second order
 which is a symmetric operator
 on the Hilbert space $L^2(X,d \mu)$. 

\begin{example}
\label{ex:Rpq}
Let $(M,g)$ be the standard flat pseudo-Riemannian manifold:
\[
 {\mathbb{R}}^{p,q}
  :=({\mathbb{R}}^{p+q}, dx_{1}^2 + \cdots + dx_{p}^2 - dx_{p+1}^2 -\cdots -dx_{p+q}^2 ).  
\]
Then the Laplacian takes the form 
\[
   \Delta_{{\mathbb{R}}^{p,q}}
   =
   \frac{\partial^2}{\partial x_{1}^2}+ \cdots + \frac{\partial^2}{\partial x_{p}^2}
-\frac{\partial^2}{\partial x_{p+1}^2}-\cdots -\frac{\partial^2}{\partial x_{p+q}^2}.  
\]
\end{example}
In general, 
 $\Delta_M$ is an elliptic differential operator
 if $(M,g)$ is Riemannian,  
and is a hyperbolic operator
 if $(M,g)$ is Lorentzian.

\subsection{Homogeneous pseudo-Riemannian manifolds}
\label{subsec:3B}

A typical example
 of pseudo-Riemannian manifolds $X$
 with \lq\lq{large}\rq\rq\ isometry groups
 is semisimple symmetric spaces,
 for which the infinitesimal classification
 was accomplished
 by M. Berger
 in 1950s.  
In this case,
 $X$ is given as a homogeneous space $G/H$
 where $G$ is a semisimple Lie group 
 and $H$ is an open subgroup of the fixed point group
 $G^{\sigma}=\{g \in G: \sigma g=g\}$
 for some involutive automorphism $\sigma$ of $G$.  
In particular,
 $G \supset H$ are a pair of reductive Lie groups.

More generally, 
 we say $G/H$ is a {\it{reductive homogeneous space}}
 if $G \supset H$ are a pair of real reductive algebraic groups.  
Then we have the following:
\begin{proposition}
\label{prop:GHg}
Any reductive homogeneous space
 $X=G/H$
 carries a pseudo-Riemannian structure
 such that 
 $G$ acts on $X$ by isometries.  
\end{proposition}
\begin{proof}
By a theorem of Mostow,
 we can take a Cartan involution $\theta$
 of $G$
such that $\theta H=H$.  
Then $K:=G^{\theta}$
 is a maximal compact subgroup of $G$, 
 and $H \cap K$ is that of $H$.  
Let ${\mathfrak {g}}= {\mathfrak {k}}+{\mathfrak {p}}$
 be the corresponding Cartan decomposition 
 of the Lie algebra ${\mathfrak {g}}$ of $G$.  
Take an $\operatorname{Ad}(G)$-invariant nondegenerate symmetric 
 bilinear form $\langle \, , \, \rangle$
 on ${\mathfrak {g}}$ such that 
 $\langle \, , \, \rangle|_{{\mathfrak {k}} \times {\mathfrak {k}}}$
 is negative definite, 
 $\langle \, , \, \rangle|_{{\mathfrak {p}} \times {\mathfrak {p}}}$
 is positive definite, 
 and  ${\mathfrak {k}}$ and ${\mathfrak {p}}$ are orthogonal to each other.  
(If $G$ is semisimple, 
then we may take $\langle \, , \, \rangle$ to be the Killing form of ${\mathfrak {g}}$.  )

Since $\theta H=H$, 
 the Lie algebra ${\mathfrak {h}}$ of $H$ is decomposed
 into a direct sum ${\mathfrak {h}}=({\mathfrak {h}} \cap {\mathfrak {k}})+({\mathfrak {h}} \cap {\mathfrak {p}})$, 
 and therefore the bilinear form $\langle \, , \, \rangle$ is non-degenerate
 when restricted to ${\mathfrak {h}}$.  
Then $\langle \, , \, \rangle$ induces an $\operatorname{Ad}(H)$-invariant 
nondegenerate symmetric bilinear form 
 $\langle \, , \, \rangle_{{\mathfrak {g}}/{\mathfrak {h}}}$
 on the quotient space ${\mathfrak {g}}/{\mathfrak {h}}$, 
 with which we identify 
 the tangent space $T_o(G/H)$
at the origin $o=eH \in G/H$.  
Since the bilinear form $\langle \, , \, \rangle_{{\mathfrak {g}}/{\mathfrak {h}}}$
 is $\operatorname{Ad}(H)$-invariant,
 the left translation of this form
 is well-defined and gives a pseudo-Riemannian structure $g$ on $G/H$
 of signature 
 $(\dim {\mathfrak {p}}/{\mathfrak {h}} \cap {\mathfrak {p}}, 
   \dim {\mathfrak {k}}/{\mathfrak {h}} \cap {\mathfrak {k}})$.  
By the construction, 
 the group $G$ acts on the pseudo-Riemannian manifold $(G/H, g)$
 by isometries.  
\qed
\end{proof}

\subsection{Pseudo-Riemannian manifolds
 with constant curvature, 
Anti-de Sitter manifolds}
\label{subsec:3C}
Let $Q_{p,q}(x):=x_{1}^2 + \cdots + x_{p}^2 - x_{p+1}^2-\cdots -x_{p+q}^2$
 be a quadratic form 
 on ${\mathbb{R}}^{p+q}$
 of signature $(p,q)$, 
 and we denote by $O(p,q)$ the indefinite orthogonal group 
 preserving the form $Q_{p,q}$.  
We define two hypersurfaces $M_{\pm}^{p,q}$ in ${\mathbb{R}}^{p+q}$
 by 
\[
M_{\pm}^{p,q}
:=
\{
 x \in {\mathbb{R}}^{p+q}:
Q_{p,q}(x)= \pm 1
\}.  
\]
By switching $p$ and $q$, 
 we have an obvious diffeomorphism
\[
M_{+}^{p,q} \simeq M_{-}^{q,p}.  
\]
The flat pseudo-Riemannian structure ${\mathbb{R}}^{p,q}$
 (Example \ref{ex:Rpq})
 induces a pseudo-Riemannian structure on the hypersurface $M_{+}^{p,q}$
 of signature $(p-1,q)$
 with constant curvature 1, 
and that on $M_{-}^{p,q}$ 
 of signature $(p,q-1)$ with constant curvature $-1$.  

The natural action
 of the group $O(p,q)$ 
on ${\mathbb{R}}^{p,q}$
 induces an isometric 
 and transitive action on the hypersurfaces $M_{\pm}^{p,q}$, 
 and thus they are expressed
 as homogeneous spaces:
\[
M_{+}^{p,q} \simeq O(p,q)/O(p-1,q), 
\quad
M_{-}^{p,q} \simeq O(p,q)/O(p,q-1), 
\]
giving examples
 of pseudo-Riemannian homogeneous spaces
 as in Proposition \ref{prop:GHg}.

The {\it{anti-de Sitter space}} 
$
   \operatorname{Ad S}^n=M_-^{n-1,2}
$
 is a model space
 for $n$-dimensional Lorentzian manifolds
 of constant negative sectional curvature, 
 or {\it{anti-de Sitter $n$-manifolds}}.  
This is a Lorentzian analogue
 of the real hyperbolic space
 $H^n$.  
For the convenience of the reader,
 we list model spaces
 of Riemannian and Lorentzian manifolds
 with constant positive, 
 zero,
 and negative curvatures.

Riemannian manifolds with constant curvature: 
\begin{alignat*}{3}
S^{n}&=\ M_+^{n+1,0} \simeq && O(n+1)/O(n) &&:\text{standard sphere}, 
\\
{\mathbb{R}}^{n}&  &&  &&:\text{Euclidean space}, 
\\
H^{n}&=\ M_-^{n,1} \simeq && O(1,n)/O(n) &&:\text{hyperbolic space}, 
\intertext{Lorentzian manifolds with constant curvature:}
\operatorname{d S}^{n}&=\ M_+^{n,1} \simeq && O(n,1)/O(n-1,1) &&:\text{de Sitter space}, 
\\
{\mathbb{R}}^{n-1,1}&  &&  &&:\text{Minkowski space}, 
\\
\operatorname{AdS}^{n}&=\ M_-^{n-1,2} \simeq && O(2,n-1)/O(1,n-1) &&:\text{anti-de Sitter space}, 
\end{alignat*}
\section{Discontinuous groups for pseudo-Riemannian manifolds}
\label{sec:4}

\subsection{Existence problem
 of compact Clifford--Klein forms}
\label{subsec:4A}
Let $H$ be a closed subgroup 
 of a Lie group $G$, 
 and $X=G/H$, 
 and $\Gamma$ a discrete subgroup of $G$.  
If $H$ is compact,
 then the double coset space $\Gamma \backslash G/H$
 becomes a $C^{\infty}$-manifold 
 for any torsion-free discrete subgroup $\Gamma$ of $G$.  
However, 
 we have to be careful 
 for noncompact $H$, 
 because not all discrete subgroups
 acts properly discontinuously on $G/H$, 
 and $\Gamma \backslash G/H$ may not be Hausdorff
 in the quotient topology.  
We illustrate this feature
 by two general results: 
\begin{fact}
\label{fact:4.1}
\begin{enumerate}
\item[{\rm{(1)}}]
{\rm{(Moore's ergodicity theorem \cite{Moore})}}
\enspace
Let $G$ be a simple Lie group,
 and $\Gamma$ a lattice.  
Then $\Gamma$ acts ergodically 
 on $G/H$
 for any noncompact closed subgroup $H$.  
In particular,
 $\Gamma \backslash G/H$ is non-Hausdorff.  
\item[{\rm{(2)}}]
{\rm{(Calabi--Markus phenomenon (\cite{CM, kob89}))}}
Let $G$ be a reductive Lie group,
 and $\Gamma$ an infinite discrete subgroup.  
Then $\Gamma \backslash G/H$ is non-Hausdorff 
 for any reductive subgroup $H$
 with 
$
   \operatorname{rank}_{\mathbb{R}} G
  = 
   \operatorname{rank}_{\mathbb{R}} H.  
$
\end{enumerate}
\end{fact}

In fact,
 determining which groups act properly discontinuously 
 on reductive homogeneous spaces $G/H$
 is a delicate problem, 
 which was first considered 
 in full generality 
 by the author;
 we refer to \cite[Section 3.2]{ky05}
 for a survey.

Suppose now a discrete subgroup $\Gamma$ acts
 properly discontinuously and freely 
 on $X=G/H$.  
Then the quotient space
\[
  X_{\Gamma} := \Gamma \backslash X
              \simeq \Gamma \backslash G/H
\]
carries a $C^{\infty}$-manifold structure
 such that the quotient map $p:X \to X_{\Gamma}$ is a covering, 
 through which $X_{\Gamma}$ inherits
 any $G$-invariant local geometric structure on $X$.  
We say $\Gamma$ is a {\it{discontinuous group for }}$X$
 and $X_{\Gamma}$ is a {\it{Clifford--Klein form}}
 of $X=G/H$.

\begin{example}
\begin{enumerate}
\item[{\rm{(1)}}]
If $X=G/H$ is a reductive homogeneous space,
 then any Clifford--Klein form $X_{\Gamma}$ carries
 a pseudo-Riemannian structure
 by Proposition \ref{prop:GHg}.  
\item[{\rm{(2)}}]
If $X=G/H$ is a semisimple symmetric space,
 then any Clifford--Klein form $X_{\Gamma}=\Gamma \backslash G/H$
 is a pseudo-Riemannian locally symmetric space, 
 namely,
 the (local) geodesic symmetry 
 at every $p \in X_{\Gamma}$
 with respect to the Levi-Civita connection
 is locally isometric.  
\end{enumerate}
\end{example}

By {\it{space forms}}, 
 we mean pseudo-Riemannian manifolds of constant sectional curvature.  
They are examples
 of pseudo-Riemannian locally symmetric spaces.  
For simplicity, 
 we shall assume
 that they are geodesically complete.  

\begin{example}
Clifford--Klein forms
 of $M_+^{p+1,q}=O(p+1,q)/O(p,q)$
(respectively,
 $M_-^{p,q+1}=O(p,q+1)/O(p,q)$)
 are pseudo-Riemannian space forms
 of signature $(p,q)$
 with positive 
 (respectively,
 negative)
 curvature.  
Conversely,
 any (geodesically complete) pseudo-Riemannian space form
 of signature $(p,q)$ 
 is of this form
 as far as $p \ne 1$
 for positive curvature
 or $q \ne 1$
 for negative curvature.  
\end{example}

A general question
 for reductive homogeneous spaces $G/H$ is:
\begin{question}
\label{q:1}
Does compact Clifford--Klein forms
 of $G/H$ exist?
\end{question}
or equivalently, 
\begin{question}
\label{q:2}
Does there exist a discrete subgroup 
 $\Gamma$ of $G$ acting cocompactly
 and properly discontinuously 
 on $G/H$?
\end{question}

This question has an affirmative answer
 if $H$ is compact 
 by a theorem of Borel.  
In the general setting
 where $H$ is noncompact,
 the question relates with a 
 \lq\lq{global theory}\rq\rq\
 of pseudo-Riemannian geometry:
{\sl{how local pseudo-Riemannian homogeneous structure
 affects the global nature
 of manifolds?}}  
A classic example is 
{\it{space form problem}}
 which asks the global properties
({\it{e.g.}} compactness,
 volume, 
fundamental groups, 
{\it{etc.}})
 of a pseudo-Riemannian manifold
 of constant curvature
 (local property).  
The study of discontinuous groups 
 for 
$
  M_+^{p+1,q}
$
and 
$
  M_-^{p,q+1}
$
shows the following results
 in pseudo-Riemannian space forms
 of signature $(p,q)$:
\begin{fact}
Space forms of positive curvature are
\begin{enumerate}
\item[{\rm{(1)}}]
always closed if $q=0$, 
{\it{i.e.}}, 
 sphere geometry in the Riemannian case;
\item[{\rm{(2)}}]
never closed 
if $p \ge q>0$, 
 in particular,
 if $q=1$
(de Sitter geometry in the Lorentzian case
\cite{CM}).
\end{enumerate}
\end{fact}
The phenomenon in the second statement is called the 
{\it{Calabi--Markus phenomenon}}
 (see Fact \ref{fact:4.1} (2)
 in the general setting).  
\begin{fact}
\label{fact:4.7}
Compact space forms of negative curvature exist
\begin{enumerate}
\item[{\rm{(1)}}]
for all dimensions
 if $q=0$, 
{\it{i.e.}}, 
hyperbolic geometry in the Riemannian case;  
\item[{\rm{(2)}}]
for odd dimensions 
 if $q=1$, 
{\it{i.e.}}, 
anti-de Sitter geometry in the Lorentzian case;
\item[{\rm{(3)}}]
for $(p,q)=(4m,3)$
 $(m \in {\mathbb{N}})$
 or $(8,7)$.  
\end{enumerate}
\end{fact}
See \cite[Section 4]{ky05} 
 for the survey
 of the space form problem
 in pseudo-Riemannian geometry
 and also of Question \ref{q:1}
 for more general $G/H$.

\vskip 0.8pc
A large and important class 
 of Clifford--Klein forms 
 $X_{\Gamma}$
 of a reductive homogeneous space $X=G/H$
 is constructed as follows 
 (see \cite{kob89}).  
\begin{definition}
{\rm{
A quotient $X_{\Gamma}= {\Gamma}\backslash X$
 of $X$
 by a discrete subgroup $\Gamma$ of $G$
 is called 
 {\it{standard}}
 if $\Gamma$ is contained 
in some reductive subgroup $L$
 of $G$ acting properly on $X$.  
}}
\end{definition}

If a subgroup $L$ acts properly on $G/H$, 
 then any discrete subgroup of $\Gamma$ acts properly discontinuously on 
 $G/H$.  
A handy criterion
 for the triple $(G,H,L)$
 of reductive groups
 such that $L$ acts properly on $G/H$
 is proved in \cite{kob89}, 
 as we shall recall below.  
Let $G=K \exp \overline{{\mathfrak {a}}_+} K$
 be a Cartan decomposition, 
where ${\mathfrak {a}}$ is a maximal abelian subspace
 of ${\mathfrak {p}}$
 and $\overline{{\mathfrak {a}}_+}$ is the dominant Weyl chamber
 with respect to a fixed positive system
 $\Sigma^+({\mathfrak{g}}, {\mathfrak{a}})$.  
This defines a map
 $\mu: G \to \overline{{\mathfrak {a}}_+}$
 ({\it{Cartan projection}})
by 
\[
   \mu(k_1 e^X k_2)
   =
   X
\quad
\text{for }
  k_1, k_2 \in K
\text{ and }
X\in {\mathfrak{a}}.  
\]
It is continuous, 
proper and surjective.  
If $H$ is a reductive subgroup, 
then there exists $g \in G$
 such that $\mu(g H g^{-1})$ is given by the intersection 
 of $\overline{{\mathfrak {a}}_+}$
 with a subspace
 of dimension $\operatorname{rank}_{\mathbb{R}} H$.  
By an abuse of notation,
 we use the same $H$ 
 instead of $g H g^{-1}$. 
With this convention, 
we have:
\begin{pcriterion}
[{\cite{kob89}}]
\label{pcr:4.9}
$L$ acts properly
 on $G/H$
 if and only if $\mu(L) \cap \mu(H)=\{0\}$.  
\end{pcriterion} 
By taking a lattice $\Gamma$ of such $L$, 
we found 
 a family of pseudo-Riemannian locally symmetric spaces
 $X_{\Gamma}$ in \cite{kob89, ky05}.  
The list of symmetric spaces 
 admitting standard Clifford--Klein forms of finite volume 
 (or compact forms)
 include $M_-^{p,q+1}=O(p,q+1)/O(p,q)$
 with $(p,q)$ 
 satisfying the conditions in Fact \ref{fact:4.7}.  
Further, 
 by applying Properness Criterion \ref{pcr:4.9}, 
 Okuda \cite{Ok} gave examples
 of pseudo-Riemannian locally symmetric spaces
 $\Gamma \backslash G/H$
 of infinite volume
 where $\Gamma$ is isomorphic to the fundamental group $\pi_1(\Sigma_g)$
 of a compact Riemann surface $\Sigma_g$
 with $g \ge 2$.

For the construction of stable spectrum 
 on $X_{\Gamma}$
 (see Theorem \ref{thm:6.4}
 and Theorem \ref{thm:7.1} (2)
 below),
 we introduced 
 in \cite[Section 1.6]{Adv16} the following concept:
\begin{definition}
{\rm{
A discrete subgroup $\Gamma$
 of $G$
 acts {\it{strongly properly discontinuously}}
 (or {\it{sharply}})
 on $X=G/H$
 if there exists $C$, $C'>0$
 such that for all $\gamma \in \Gamma$, 
\[
   d (\mu(\gamma), \mu(H)) \ge C\| \mu(\gamma) \| - C'.  
\]
Here $d(\cdot, \cdot)$ is a distance
 in ${\mathfrak {a}}$
 given by a Euclidean norm
 $\| \cdot \|$ 
 which is invariant 
 under the Weyl group of the restricted root system
 $\Sigma({\mathfrak{g}}, {\mathfrak{a}})$.  
We say the positive number $C$ is {\it{the first sharpness constant}} for $\Gamma$.  
}}
\end{definition}

If a reductive subgroup $L$ acts properly
 on a reductive homogeneous space $G/H$, 
then the action
 of a discrete subgroup $\Gamma$
 of $L$ is strongly properly discontinuous
 (\cite[Example 4.10]{Adv16}).

\subsection{Deformation of Clifford--Klein forms}
\label{subsec:4B}

Let $G$ be a Lie group and $\Gamma$ a finitely generated group.  
We denote by $\operatorname{Hom}(\Gamma,G)$ the set of 
 all homomorphisms of $\Gamma$ to $G$ topologized by pointwise convergence.  
By taking a finite set
 $\{\gamma_1, \cdots, \gamma_k\}$
 of generators of $\Gamma$, 
 we can identify $\operatorname{Hom}(\Gamma,G)$
 as a subset of the direct product 
 $G \times \cdots \times G$
 by the inclusion:
\begin{equation}
\label{eqn:HG}
   \operatorname{Hom}(\Gamma,G)
  \hookrightarrow
   G \times \cdots \times G, 
  \quad
  \varphi \mapsto (\varphi(\gamma_1), \cdots, \varphi(\gamma_k)). 
\end{equation}
If $\Gamma$ is finitely presentable, 
 then $\operatorname{Hom}(\Gamma,G)$
 is realized as a real analytic variety
 via \eqref{eqn:HG}.  

Suppose $G$ acts continuously 
 on a manifold $X$.  
We shall take $X=G/H$
 with noncompact closed subgroup $H$ later.  
Then not all discrete subgroups
 act properly discontinuously
 on $X$
 in this general setting.  
The main difference of the following definition
 of the author \cite {kob93} in the general case from 
that of Weil \cite{Weil}
 is a requirement 
 of proper discontinuity.  
\begin{align}
\label{eqn:R}
  R(\Gamma, G;X)
  :=  
  \{ & \varphi \in \operatorname{Hom}(\Gamma,G):
    \text{$\varphi$ is injective,} 
\\
    &\text{and }\varphi(\Gamma) \text{ acts properly discontinuously and freely on }G/H\}.  
\notag
\end{align}
Suppose now $X=G/H$ for a closed subgroup $H$.  
Then the double coset space $\varphi(\Gamma)\backslash G/H$
 forms a family of manifolds 
 that are locally modelled on $G/H$ 
 with parameter $\varphi \in R(\Gamma, G;X)$.  
To be more precise on \lq\lq{parameter}\rq\rq, 
 we note 
 that the conjugation
 by an element
 of $G$ induces an automorphism
 of $\operatorname{Hom}(\Gamma,G)$
 which leaves $R(\Gamma, G;X)$
 invariant.  
Taking these unessential deformations into account,
 we define the {\it{deformation space}} 
 ({\it{generalized Teichm{\"u}ller space}})
 as the quotient set
\[
  {\mathcal {T}}(\Gamma, G;X)
  :=
  R(\Gamma, G;X)/G.  
\]
\begin{example}
\begin{enumerate}
\item[{\rm{(1)}}]
Let $\Gamma$ be the surface group 
 $\pi_1(\Sigma_g)$
 of genus $g \ge 2$, 
$G=PSL(2,{\mathbb{R}})$, 
 $X=H^2$
 (two-dimensional hyperbolic space).  
Then ${\mathcal{T}}(\Gamma, G;X)$
 is the classical Teichm{\"u}ller space, 
 which is of dimension $6g-6$.  
\item[{\rm{(2)}}]
$G={\mathbb{R}}^n$, 
 $X={\mathbb{R}}^n$, 
 $\Gamma={\mathbb{Z}}^n$.  
Then ${\mathcal{T}}(\Gamma, G;X) \simeq GL(n,{\mathbb{R}})$
 (see \eqref{eqn:TGL} below).
\item[{\rm{(3)}}]
$G=SO(2,2)$,
 $X=\operatorname{Ad S}^3$, 
and $\Gamma=\pi_1(\Sigma_g)$.  
Then ${\mathcal{T}}(\Gamma, G;X)$ is of dimension $12g-12$
 (see \cite[Section 9.2]{Adv16} and references therein).  
\end{enumerate}
\end{example}
\begin{remark}
There is a natural isometry 
 between $X_{\varphi(\Gamma)}$ and $X_{\varphi(g \Gamma g^{-1})}$.  
Hence,
the set  $\operatorname{Spec}_d(X_{\varphi(\Gamma)})$
 of $L^2$-eigenvalues 
 is independent of the conjugation
 of $\varphi \in R(\Gamma, G;X)$ by an element  of $G$.
By an abuse of notation
 we shall write $\operatorname{Spec}_d(X_{\varphi(\Gamma)})$
 for $\varphi \in {\mathcal{T}}(\Gamma, G;X)$
when we deal with Problem \ref{prob:B}
 of Section \ref{sec:2}.
\end{remark}

\section{Spectrum on ${\mathbb{R}}^{p,q}/{\mathbb{Z}}^{p+q}$
 and Oppenheim conjecture}
\label{sec:5}

This section gives an elementary but inspiring observation
 of spectrum on flat pseudo-Riemannian manifolds.  
 
\subsection{Spectrum of ${\mathbb{R}}^{p,q}/\varphi({\mathbb{Z}}^{p+q})$}
\label{subsec:5A}

Let $G={\mathbb{R}}^n$ and $\Gamma={\mathbb{Z}}^n$.  
Then the group homomorphism
 $\varphi:\Gamma \to G$ is uniquely determined
 by the image $\varphi(\vec{e}_j)$
 ($1 \le j \le n$)
 where $\vec e_1$, $\cdots$, $\vec e_n$ $ \in {\mathbb{Z}}^n$
 are the standard basis, 
 and thus we have a bijection
\begin{equation}
\label{eqn:HMnR}
\operatorname{Hom}(\Gamma, G) 
\overset \sim \leftarrow
M(n,{\mathbb{R}}),
\quad
\varphi_g 
\reflectbox{$\mapsto$}
g
\end{equation}
by 
$
  \varphi_g(\vec m):=g \vec m
$
for 
$
\vec m \in {\mathbb{Z}}^n,
$
or equivalently, 
by 
$
  g=(\varphi_g(\vec{e}_1), \cdots, \varphi_g(\vec{e}_n)).  
$

Let $\sigma \in {\operatorname{Aut}}(G)$ be defined
 by $\sigma(\vec x):=-\vec x$.
Then $H:=G^{\sigma}=\{0\}$
 and $X:=G/H \simeq {\mathbb{R}}^n$ is a symmetric space.  
The discrete group $\Gamma$ acts properly discontinuously
 on $X$
 via $\varphi_g$
 if and only if 
 $g \in GL(n,{\mathbb{R}})$.  
Moreover,
since $G$ is abelian, 
 $G$ acts trivially
 on $\operatorname{Hom}(\Gamma, G)$
 by conjugation, 
 and therefore the deformation space
 ${\mathcal{T}}(\Gamma, G;X)$
 identifies with 
$R(\Gamma, G;X)$.  
Hence we have a natural bijection 
between the two subsets
 of \eqref{eqn:HMnR}:
\begin{equation}
\label{eqn:TGL}
  {\mathcal{T}}(\Gamma, G;X) \overset \sim \leftarrow GL(n,{\mathbb{R}}).  
\end{equation}

Fix $p$, $q$ $\in {\mathbb{N}}$
 such that $p+q=n$, 
and we endow $X \simeq {\mathbb{R}}^n$
 with the standard flat indefinite metric 
 ${\mathbb{R}}^{p,q}$
 (see Example \ref{ex:Rpq}).  
Let us determine 
 $\operatorname{Spec}_d(X_{\varphi_g(\Gamma)})
  \simeq
  \operatorname{Spec}_d({\mathbb{R}}^{p,q}/\varphi_g({\mathbb{Z}}^n))$
 for $g \in GL(n,{\mathbb{R}}) \simeq {\mathcal{T}}(\Gamma, G;X)$.

For this, 
 we define a function
 on $X={\mathbb{R}}^n$ by
\[
   f_{\vec m}(\vec x)
:=
  \exp
  (2 \pi \sqrt{-1}\, {}^{t\!} \vec m g^{-1} \vec x)
\qquad
(\vec x \in {\mathbb{R}}^n)
\] 
for each $\vec m \in {\mathbb{Z}}^n$
 where $\vec x$ and $\vec m$ are regarded
 as column vectors.  
Clearly, 
 $f_{\vec m}$
 is $\varphi_g(\Gamma)$-periodic
 and defines a real analytic function
 on $X_{\varphi_g(\Gamma)}$.  
Furthermore,
 $f_{\vec m}$ is an eigenfunction
 of the Laplacian $\Delta_{{\mathbb{R}}^{p,q}}$:
\[
   \Delta_{{\mathbb{R}}^{p,q}} f_{\vec m}
  = - 4 \pi^2 Q_{g^{-1} I_{p,q} {}^{t\!} g^{-1}}(\vec m) f_{\vec m}, 
\]
where, for a symmetric matrix $S \in M(n,{\mathbb{R}})$,
 $Q_S$ denotes the quadratic form 
 on ${\mathbb{R}}^n$
 given by 
\[
 Q_S(\vec y)
:=
{}^{t\!}\vec y S \vec y
\qquad
  \text{for } \vec y \in {\mathbb{R}}^n.  
\]
Since $\{f_{\vec m}: \vec m \in {\mathbb{Z}}^n \}$
 spans a dense subspace
 of $L^2(X_{\varphi_g(\Gamma)})$, 
we have shown:

\begin{proposition}
\label{prop:SpecRn}
For any $g \in GL(n,{\mathbb{R}}) \simeq {\mathcal{T}}(\Gamma, G;X)$, 
\[
\operatorname{Spec}_d(X_{\varphi_g(\Gamma)})
=
\{-4 \pi^2 Q_{g^{-1} I_{p,q} {}^{t\!}g^{-1}}
 (\vec m): \vec m \in {\mathbb{Z}}^n\}.  
\]
\end{proposition}
Here are some observation 
 in the $n=1,2$ cases.  
\begin{example}
\label{ex:short}
Let $n=1$
 and $(p,q)=(1,0)$.  
Then $\operatorname{Spec}_d(X_{\varphi_g(\Gamma)})
=\{- 4 \pi^2 m^2 /g^2:m \in {\mathbb{Z}}\}$
 for $g \in {\mathbb{R}}^{\times}\simeq GL(1,{\mathbb{R}})$
by Proposition \ref{prop:SpecRn}.  
Thus the smaller the period $|g|$ is, 
 the larger the absolute value
 of the eigenvalue
 $|-4 \pi^2 m^2/g^2|$ becomes
 for each fixed $m \in {\mathbb{Z}} \setminus \{0\}$.  
This is thought of as a mathematical model
 of a music instrument
 for which shorter strings produce a higher pitch
 than longer strings
 (see Introduction).  
\end{example}

\begin{example}
\label{ex:n2}
Let $n=2$ and $(p,q)=(1,1)$.  
Take $g=I_2$, 
 so that $\varphi_g(\Gamma)={\mathbb{Z}}^2$
 is the standard lattice.  
Then the $L^2$-eigenspace
 of the Laplacian $\Delta_{{\mathbb{R}}^{1,1}/{\mathbb{Z}}^2}$
 for zero eigenvalue contains 
$
  W:=\{ \psi(x-y):\psi \in L^2({\mathbb{R}}/{\mathbb{Z}})\}.  
$
Since $W$ is infinite-dimensional and 
 $W \not \subset C^{\infty}({\mathbb{R}}^2/{\mathbb{Z}}^2)$, 
 the third and fourth statements of Fact \ref{fact:RLop}
 fail in this pseudo-Riemannian setting.  
\end{example}

By the explicit description of $\operatorname{Spec}_d(X_{\varphi(\Gamma)})$
 for all $\varphi \in {\mathcal{T}}(\Gamma, G;X)$
 in Proposition \ref{prop:SpecRn}, 
 we can also tell the behaviour
 of $\operatorname{Spec}_d(X_{\varphi(\Gamma)})$
under deformation of $\Gamma$
 by $\varphi$.  
Obviously,
 any constant function on $X_{\varphi(\Gamma)}$
 is an eigenfunction
of the Laplacian 
 $\Delta_{X_{\varphi(\Gamma)}}=\Delta_{{\mathbb{R}}^{p,q}}/\varphi({\mathbb{Z}}^{p+q})$
 with eigenvalue zero.  
We see that 
 this is the unique stable $L^2$-eigenvalue
 in the flat compact manifold:
\begin{corollary}
[non-existence of stable eigenvalues]
\label{cor:varyRn}
Let $n=p+q$
 with $p,q \in {\mathbb{N}}$.   
For any open subset $V$
 of ${\mathcal{T}}(\Gamma, G;X)$, 
\[
  \bigcap_{\varphi \in V} \operatorname{Spec}_d(X_{\varphi(\Gamma)})
 = \{0\}.  
\]
\end{corollary}

\subsection{Oppenheim's conjecture and stability of spectrum}
\label{subsec:5B}

In 1929, 
Oppenheim \cite{O29}
 raised a question 
 about the distribution
 of an indefinite quadratic forms
 at integral points.  
The following theorem, 
 referred to as Oppenheim's conjecture, 
 was proved by Margulis
(see \cite{M00}
 and references therein).  
\begin{fact}
[Oppenheim's conjecture]
Suppose $n \ge 3$
 and $Q$ is a real nondegenerate indefinite quadratic form
 in $n$ variables.  
Then either $Q$ is proportional to a form
 with integer coefficients
 (and thus $Q({\mathbb{Z}}^n)$
 is discrete in ${\mathbb{R}}$), 
 or $Q({\mathbb{Z}}^n)$ is dense in ${\mathbb{R}}$.  
\end{fact}
Combining this with Proposition \ref{prop:SpecRn}, 
 we get the following.  

\begin{theorem}
\label{thm:Oppen}
Let $p+q =n$, 
 $p \ge 2$, $q \ge 1$, 
$G={\mathbb{R}}^n$, 
 $X = {\mathbb{R}}^{p,q}$
 and $\Gamma = {\mathbb{Z}}^n$.  
We define an open dense subset $U$
 of ${\mathcal{T}}(\Gamma, G;X) \simeq GL(n,{\mathbb{R}})$
 by
\[
  U:=\{g \in GL(n,{\mathbb{R}})
       :
       g^{-1} I_{p,q} {}^{t\!}g^{-1}\text{ is not proportional 
       to an element of $M(n,{\mathbb{Z}})$}.
\}
\]
Then the set $\operatorname{Spec}_d(X_{\varphi(\Gamma)})$
 of $L^2$-eigenvalues
 of the Laplacian 
 is dense in ${\mathbb{R}}$
 if and only if $\varphi \in U$.  
\end{theorem}
Thus the fifth statement of Fact \ref{fact:RLop}
 for compact Riemannian manifolds
 do fail in the pseudo-Riemannian case.

\section{Main results---sound of anti-de Sitter manifolds}
\label{sec:6}

\subsection{Intrinsic sound of anti-de Sitter manifolds}
\label{subsec:6A}
In general,
 it is not clear 
 whether the Laplacian $\Delta_M$
 admits infinitely many $L^2$-eigenvalues
 for compact pseudo-Riemannian manifolds.  
For anti-de Sitter 3-manifolds,
 we proved in \cite[Theorem 1.1]{Adv16}:
\begin{theorem}
\label{thm:6.1}
For any compact anti-de Sitter 3-manifold $M$, 
 there exist infinitely many $L^2$-eigenvalues
 of the Laplacian $\Delta_M$.  
\end{theorem}

In the abelian case,
 it is easy to see
 that compactness of $X_{\Gamma}$
 is necessary
 for the existence of $L^2$-eigenvalues:

\begin{proposition}
Let $G={\mathbb{R}}^{p+q}$, 
 $X = {\mathbb{R}}^{p,q}$, 
 $\Gamma = {\mathbb{Z}}^k$, 
 and $\varphi \in R(\Gamma, G;X)$.  
Then 
$
  \operatorname{Spec}_d(X_{\varphi(\Gamma)}) \ne \emptyset
$
if and only if 
 $X_{\varphi(\Gamma)}$ is compact,
 or equivalently,
 $k=p+q$.  
\end{proposition}
However, 
 anti-de Sitter 3-manifolds $M$ admit
 infinitely many $L^2$-eigenvalues
 even when $M$ is of infinite-volume
 (see \cite[Theorem 9.9]{Adv16}):

\begin{theorem}
\label{thm:6.3}
For any finitely generated discrete subgroup $\Gamma$
 of $G=SO(2,2)$
acting properly discontinuously
 and freely
 on $X=\operatorname{Ad S}^3$,
\[
\operatorname{Spec}_d(X_{\Gamma})
\supset
\{
l(l-2)
:
l \in {\mathbb{N}}, 
l \ge 10 C^{-3}
\}
\]
where $C \equiv C(\Gamma)$ is the first sharpness constant of $\Gamma$.  
\end{theorem}
The above $L^2$-eigenvalues
 are stable in the following sense: 
\begin{theorem}
[stable $L^2$-eigenvalues]
\label{thm:6.4}
Suppose that $\Gamma \subset G=SO(2,2)$ and $M= \Gamma \backslash \operatorname{Ad S}^3$
 is a compact standard anti-de Sitter 3-manifold.  
Then there exists a neighbourhood $U \subset \operatorname{Hom}(\Gamma,G)$
 of the natural inclusion
 with the following two properties:
\begin{equation}
\label{eqn:s1}
   U \subset R(\Gamma, G;\operatorname{Ad S}^3), 
\end{equation}
\begin{equation}
\label{eqn:s2}
\# (\bigcap_{\varphi \in U}
 \operatorname{Spec}_d(X_{\Gamma}))
 = \infty.
\end{equation}
\end{theorem}
The first geometric property \eqref{eqn:s1} asserts
 that a small deformation 
of $\Gamma$ keeps proper discontinuity,
 which was conjectured
 by Goldman \cite{Go} in the $\operatorname{Ad S}^3$ setting,
 and proved affirmatively
 in \cite{kob98}.  
Theorem \ref{thm:6.4} was proved in \cite[Corollary 9.10]{Adv16}
 in a stronger form 
({\it{e.g.}}, 
without assuming \lq\lq{standard}\rq\rq\ condition).

Figuratively speaking,
 Theorem \ref{thm:6.4} says 
 that compact anti-de Sitter manifolds 
 have \lq\lq{intrinsic sound}\rq\rq\
which is stable under any small deformation
 of the anti-de Sitter structure.  
This is a new phenomenon 
 which should be in sharp contrast to the abelian case
 (Corollary \ref{cor:varyRn})
 and the Riemannian case below:

\begin{fact}
[{see \cite[Theorem 5.14]{W}}]
\label{fact:6.2}
For a compact hyperbolic surface,
 no eigenvalue
 of the Laplacian above $\frac 1 4$ is constant
 on the Teichm{\"u}ller space.  
\end{fact}

We end this section 
by raising the following question
 in connection with the flat case
 (Theorem \ref{thm:Oppen}):
\begin{question}
Suppose $M$ is a compact anti-de Sitter 3-manifold.  
Find a geometric condition on $M$
 such that $\operatorname{Spec}_d(M)$ is discrete.  
\end{question}
\section{Perspectives and sketch of proof}
\label{sec:7}

The results in the previous section
 for anti-de Sitter 3-manifolds
 can be extended to more general pseudo-Riemannian locally symmetric spaces
 of higher dimension:
\begin{theorem}
[{\cite[Theorem 1.5]{Adv16}}]
\label{thm:7.1}
Let $X_{\Gamma}$ be a standard Clifford--Klein form
 of a semisimple symmetric space 
 $X=G/H$ satisfying the rank condition
\begin{equation}
\label{eqn:rank}
\operatorname{rank} G/H = \operatorname{rank} K/ H \cap K.  
\end{equation}
Then the following holds.
\begin{enumerate}
\item[{\rm{(1)}}]
There exists an explicit infinite subset $I$
 of joint $L^2$-eigenvalues
 for all the differential operators
 on $X_{\Gamma}$
 that are induced from $G$-invariant 
 differential operators on $X$.  
\item[{\rm{(2)}}]
(stable spectrum)
If $\Gamma$ is contained in a simple Lie group $L$
 of real rank one acting properly on $X=G/H$,
 then there is a neighbourhood $V \subset \operatorname{Hom}(\Gamma,G)$
 of the natural inclusion
 such that for any $\varphi \in V$, 
 the action $\varphi(\Gamma)$ on $X$ is properly discontinuous
 and the set of joint $L^2$-eigenvalues
 on $X_{\varphi(\Gamma)}$ contains the infinite set $I$.  
\end{enumerate}
\end{theorem}

\begin{remark}
We do not require $X_{\Gamma}$
 to be of finite volume in Theorem \ref{thm:7.1}.  
\end{remark}

\begin{remark}
It is plausible
 that for a general locally symmetric space $\Gamma\backslash G/H$
with $G$ reductive,
 no nonzero $L^2$-eigenvalue
 is stable under nontrivial small deformation
unless the rank condition \eqref{eqn:rank} is satisfied.  
For instance, 
 suppose $\Gamma = \pi_1(\Sigma_g)$
 with $g \ge 2$
 and $R(\Gamma, G;X) \ne \emptyset$.  
(Such semisimple symmetric space $X=G/H$ was recently 
 classified in \cite{Ok}.)
Then we expect the rank condition \eqref{eqn:rank} is
equivalent to the existence of an open subset $U$ 
 in $R(\Gamma, G;X)$
 such that 
\[
  \# (\bigcap_{\varphi \in U}
  \operatorname{Spec}_d(X_{\varphi(\Gamma)})) = \infty.  
\]
\end{remark}

It should be noted
 that not all $L^2$-eigenvalues
 of compact anti-de Sitter manifolds
 are stable under small deformation
 of anti-de Sitter structure.  
In fact, 
 we proved in \cite{KKspec} that
 there exist also countably many 
{\it{negative}}
 $L^2$-eigenvalues 
 that are NOT stable under deformation, 
 whereas the countably many stable $L^2$-eigenvalues
 that we constructed in Theorem \ref{thm:6.3}
 are all positive. 
More generally,
 we prove in \cite{KKspec}
 the following theorem 
 that include both stable and unstable 
 $L^2$-eigenvalues: 
\begin{theorem}
\label{thm:7.2}
Let $G$ be a reductive homogeneous space
 and $L$ a reductive subgroup of $G$
 such that $H \cap L$ is compact.  
Assume that the complexification $X_{\mathbb{C}}$ is 
$L_{\mathbb{C}}$-spherical.  
Then for any torsion-free discrete subgroup $\Gamma$ of $L$, 
 we have:
\begin{enumerate}
\item[{\rm{(1)}}]
the Laplacian $\Delta_{X_{\Gamma}}$ extends to a self-adjoint operator
 on $L^2(X_{\Gamma})$;
\item[{\rm{(2)}}]
$\#\operatorname{Spec}_d(X_{\Gamma})=\infty$
 if $X_{\Gamma}$ is compact.  
\end{enumerate}
\end{theorem}
By \lq\lq{$L_{\mathbb{C}}$-spherical}\rq\rq\
 we mean
that a Borel subgroup $L_{\mathbb{C}}$ has an open orbit 
in $X_{\mathbb{C}}$.  
In this case,
 a reductive subgroup 
 $L$ acts transitively on $X$ 
by \cite[Lemma 5.1]{kob94}.

Here are some examples
 of the setting of Theorem \ref{thm:7.2},
 taken from \cite[Corollary 3.3.7]{ky05}.

\begin{table}[h]
\caption{\ref{table:GHL}}
\begin{center}
\begin{tabular}{c|c|c|c}
& $G$ & $H$  & $L$ 
\\
\hline
(i)
&
$SO(2n,2)$
&
$SO(2n,1)$
&
$U(n,1)$
\\
(ii)
&
$SO(2n,2)$
&
$U(n,1)$
&
$SO(2n,1)$
\\
(iii)
&
$SU(2n,2)$
&
$U(2n,1)$
&
$Sp(n,1)$
\\
(iv)
&
$SU(2n,2)$
&
$Sp(n,1)$
&
$U(2n,1)$
\\
(v)
&
$SO(4n,4)$
&
$SO(4n,3)$
&
$Sp(1) \times Sp(n,1)$
\\
(vi)
&
$SO(8,8)$
&
$SO(8,7)$
&
$Spin(8,1)$
\\
(vii)
&
$SO(8,{\mathbb{C}})$
&
$SO(7,{\mathbb{C}})$
&
$Spin(7,1)$
\\
(viii)
&
$SO(4,4)$
&
$Spin(4,3)$
&
$SO(4,1) \times SO(3)$
\\
(ix)
&
$SO(4,3)$
&
$G_2({\mathbb{R}})$
&
$SO(4,1) \times SO(2)$
\end{tabular}
\label{table:GHL}
\end{center}
\end{table}%

Examples for Theorem \ref{thm:7.2}
 include Table \ref{table:GHL} (ii)
 for all $n \in {\mathbb{N}}$, 
whereas we need $n \in 2{\mathbb{N}}$
 in Theorem \ref{thm:7.1}
 for the rank condition \eqref{eqn:rank}.

The idea of the proof for Theorem \ref{thm:7.1} is to take an average
of a (nonperiodic) eigenfunction on $X$
with rapid decay at infinity
over $\Gamma$-orbits 
 as a generalization of Poincar{\'e} series.
Geometric ingredients of the convergence
 (respectively, nonzeroness)
 of the generalized Poincar{\'e} series
 include \lq\lq{counting $\Gamma$-orbits}\rq\rq\
 stated in Lemma \ref{lem:counting} below
 (respectively,
 the Kazhdan--Margulis theorem, 
 {\textit{cf}}.~\cite[Proposition 8.14]{Adv16}).  
Let $B(o,R)$ be a \lq\lq{pseudo-ball}\rq\rq\
 of radius $R >0$
 centered at the origin $o=e H \in X = G/H$, 
 and we set 
\[
  N(x,R):=
\# \{\gamma \in \Gamma:
\gamma \cdot x \in B(o,R)\}.  
\]
\begin{lemma}
[{\cite[Corollary 4.7]{Adv16}}]
\label{lem:counting}
\begin{enumerate}
\item[{\rm{(1)}}]
If $\Gamma$ acts properly discontinuously on $X$, 
then $N(x,R)< \infty$ for all $x \in X$
 and $R>0$.   
\item[{\rm{(2)}}]
If $\Gamma$ acts strongly properly discontinuously on $X$, 
then there exists $A_x >0$
such that 
\[
N(x,R) \le A_x \exp (\frac R C)
\quad
\text{for all } R>0, 
\]
where $C$ is the first sharpness constant of $\Gamma$.   
\end{enumerate}
\end{lemma}

The key idea of Theorem \ref{thm:7.2} is to bring
 branching laws to spectral analysis 
\cite{kob94, kob09},
namely, we consider the restriction of 
 irreducible representations of $G$ that are realized 
 in the space of functions on the homogeneous space $X=G/H$
and analyze the $G$-representations when restricted to
 the subgroup $L$. Details will be given in \cite{KKspec}.

\medskip

\begin{acknowledgement}
This article is based 
 on the talk 
 that the author delivered at the eleventh International Workshop:
Lie Theory and its Applications
in Physics
 in Varna, Bulgaria, 15-21, June, 2015.  
The author is grateful to Professor Vladimir Dobrev
 for his warm hospitality.  
\end{acknowledgement}

%
%
%

\end{document}